\documentclass{amsart}

\newtheorem{theorem}{Theorem}[section]
\usepackage{ amsmath, amsthm,  amscd, amsfonts, amssymb, graphicx, stix, mathtools, stmaryrd, color }
\newtheorem{lemma}[theorem]{Lemma}
\newtheorem{proposition}[theorem]{Proposition}

\theoremstyle{definition}
\newtheorem{definition}[theorem]{Definition}

\newtheorem{problem}[theorem]{Problem}
\theoremstyle{remark}

\numberwithin{equation}{section}

\numberwithin{equation}{section}

\begin{document}

\title{2-Local Derivations on Some $ C^* $-Algebras}


\author{Meysam Habibzadeh Fard}
\address{Department of Mathematics, \\  
University of Guilan, \\
 Guilan, \\ 
  Iran}
\curraddr{}
\email{mhabibzadeh@phd.guilan.ac.ir, meysam.habibzadeh@icloud.com}  
\thanks{}

\author{Abbas Sahleh}
\address{Department of Mathematics, \\
University of Guilan, \\ 
 Guilan, \\ 
  Iran }
\curraddr{}
\email{sahlehj@gmail.com,  sahlehjj@guilan.ac.ir}
\thanks{}



\date{20 November 2017}

\dedicatory{}

\begin{abstract}
 In this paper, we introduce the concept of trace-open projections in the second dual $ \mathcal{A}^{**} $ of a $ C^* $-algebra
  $ \mathcal{A} $, and we show that if there  is a faithful normal semi-finite trace  $ \tau $  on $ \mathcal{A}^{**} $, and 
  $ 1_{ \mathcal{A}^{**} } $ is a $ \tau $-open projection, then each 2-local  derivation $ \Delta $ from $ \mathcal{A} $ to $ \mathcal{A}^{**} $  is an inner derivartion. We also show that this theorem hold for approximately 2-local derivations when $ \mathcal{A}^{**} $ is a finite von neumann algebra. 
\end{abstract}

\maketitle

\section{Introduction}

   The notion of 2-local derivations introduced by Semrl in \cite{Semrl}, he proved that if
    $ \mathcal{H} $ is an infinite dimensional separable hilbert space, then every 2-local derivation on $\mathcal{B(H)}$ is a derivation. In \cite{Kim1}, Kim and Kim proved that every 2-local derivation on $ M_n(\mathbb{C}) $ is a derivation, this implies that if $ \mathcal{H} $ is a finite dimensional hilbert space, then every 2-local derivation on $\mathcal{B(H)}$ is a derivation.  
In \cite{Ayupov1}, Ayupov and Kudaybergenov proved that every 2-local derivation on $ \mathcal{B(H)} $ is a derivation for arbitrary Hilbert space $ \mathcal{H}$.  
In \cite{Ayupov2}, they also proved that each 2-Local derivation on a semi-finite von Neumann algebra is a derivation and finally in \cite{Ayupov3}, authors showed that this result hold for arbitrary von Neumann algebras.

In \cite{Kim2}, Kim and Kim proved that every norm continuous 2-local derivation on AF $ C^* $-algebra is a derivation. In this paper we survey about 2-local derivations of $ C^* $-algebra into it's second dual.

Section 2  is about the preliminaries.

In Section 3, we define the concept of trace-open projections of a $ C^* $-algebra $ \mathcal{A} $, and we show that if $ \tau $ is a faithful normal semi-finite trace on $ \mathcal{A}^{**} $ and $ 1_{\mathcal{A}^{**}} $  is a $ \tau $-open projection, then each 2-local derivation 
$ \Delta $ from $ \mathcal{A} $ into $ \mathcal{A}^{**} $ is a derivation. We also show that if $ C^* $-algebra $ \mathcal{A} $ is an ideal on it's second dual, then each 2-local derivation on $ \mathcal{A} $ is a derivation.

 In Section 4, we show that if the second dual of a $ C^* $-algebra $ \mathcal{A} $ is finite as a von Neumann algebra, then each approximately 2-local derivation 
$ \Delta $ from $ \mathcal{A} $ into $ \mathcal{A}^{**} $ is a derivation.

\section{Preliminaries}

Let $ \mathcal{A} $ be a Banach algebra and let $ E $ be a Banach $ \mathcal{A} $-bimodule.  A (bounded) linear map $ D : \mathcal{A} \longrightarrow E $ is a (bounded) derivation if
$$
 D(ab) = a \cdot Db + Da\cdot b \hspace{.5cm}(a,b \in \mathcal{A}).
$$
In the case that $ E = \mathcal{A} $, the derivation $ D $ is called a derivation on $ \mathcal{A} $.  
For $ a \in E $, put $ \delta_{a}(x) = [ a, x ] = a\cdot x - x\cdot a $, $ ( x \in \mathcal{A} ) $, then $ \delta_{a} $ is a bounded derivation. Such a derivation is called an inner derivation.   
 A Jordan derivation from a Banach algebra $ \mathcal{A} $ into a Banach $ \mathcal{A} $-bimodule $ E $ is a linear map $ D $ satisfying $ D(a^2) = a \cdot D(a) + D(a) \cdot a$ ($a \in \mathcal{A} $).  
 
 The following statements hold:  
 \begin{itemize}  
 \item  Each  derivation $ \delta $ from a $ C^* $-algebra $ \mathcal{A} $ to a Banach $ \mathcal{A} $-bimodule $ E $ is bounded (\cite{Ringrose}, theorem 2). 
 \item  Each derivation $ \delta $ on a von Neumann algebra $ \mathcal{A} $ is inner (\cite{Sakai}, Theorem 4.1.6).
 \item Each bounded Jordan derivation $ D $ from a $ C^{*} $-algebra $ \mathcal{A} $ into a Banach $ \mathcal{A} $-bimodule $ E $ is a  derivation (\cite{Johnson2}, Theorem 6.3).
 \item Every Jordan derivation $ D $ from a $ C^{*} $-algebra $ \mathcal{A} $ into a Banach $ \mathcal{A} $-bimodule $ E $ is  bounded
   ( \cite{Russo},  Corollary 17 ).  
 \end{itemize}


  Let $ \mathcal{A} $ be a Banach algebra and let $ \mathcal{A}^{*} $ and $ \mathcal{A}^{**} $ be its dual and its second dual spaces, respectively.  For $ x, y \in \mathcal{A} $, $ f \in \mathcal{A}^{*} $ and $ F, G \in \mathcal{A}^{**} $, define: 
  \begin{align*} 
 (F \mdsmwhtsquare G)(f) := F(G \star f),& \hspace{.5cm} (F \mdwhtdiamond G)(f) := G(f \star F)  \\ 
where \hspace{1cm} (G \star f)(x) := G(f \cdot x),& \hspace{.5cm}  (f \star F)(x) := F(x \cdot f)   \\
and \hspace{1.5cm} (f \cdot x)(y) := f(xy) ,& \hspace{.5cm}  (x \cdot f)(y) := f(yx)      
  \end{align*}  
  Then $ F \mdsmwhtsquare G,\hspace{.1cm} F \mdwhtdiamond G \in \mathcal{A}^{**} $, $ G \star f,\hspace{.1cm} f \star F  \in \mathcal{A}^{*} $ and $ f \cdot x,\hspace{.1cm} x \cdot f  \in \mathcal{A}^{*} $.    
The algebra $ \mathcal{A} $ is called Arens regular if the products $ \mdsmwhtsquare $  and $ \mdwhtdiamond $ coincide on $ \mathcal{A}^{**} $. 
 It is known that, every $ C^{*} $-algebra $ \mathcal{A} $ is Arens regular, and its second dual $ \mathcal{A}^{**} $ with multiplication $ \mdsmwhtsquare $,  is a von Neumann algebra
 (\cite{Dales2}, Corollary 3.2.37).  
 From now on, we  avoid writing Arens product $ \mdsmwhtsquare $ on  $ \mathcal{A}^{**} $ ( the second dual of a $ C^{*} $-algebra $ \mathcal{A} $ ).

 \begin{lemma} 
 Each derivation $ D $ from a $ C^{*} $-algebra $ \mathcal{A} $ into its second dual $ \mathcal{A}^{**} $ is inner.  
 \end{lemma}

 \begin{proof}
By Arens regularity of $ C^{*} $-algebra $ \mathcal{A} $, each  derivation $ D $ from  $ \mathcal{A} $ to $ \mathcal{A}^{**}  $  can be extended to a derivation $ \widetilde{D}:\mathcal{A}^{**}\longrightarrow \mathcal{A}^{**} $  (see \cite{Dales2}, Proposition 2.7.17).  Since $ \mathcal{A}^{**} $ is a von Neumann algebra, there is an element $ a \in \mathcal{A}^{**} $ such that $ \widetilde{D}(x) = [a, x] $  (\cite{Sakai}, Theorem 4.1.6) and therefore 
 $$
 D(x) = \widetilde{D}(x)|_{\mathcal{A}} = [a, x],  \hspace{.5cm}  ( x \in \mathcal{A} ). 
 $$
 \end{proof}

   A mapping $ \Delta $ from a Banach algebra $ \mathcal{A} $ into a Banach $ \mathcal{A} $-bimodule $ E $  is a bounded 2-local (respectively, approximately 2-local) derivation, if for each $ a, b \in \mathcal{A} $, there is a bounded derivation $D_{a,b}$ (respectively, a sequence of bounded derivations $\{D^{n}_{a,b}\}$) from $ \mathcal{A} $ into $ E $ such that $D(a) = D_{a,b}(a)$ and 
   $D(b) = D_{a,b}(b)$ (respectively $D(a) = \lim_{n\longrightarrow\infty} D^{n}_{a,b}(a)$ and 
     $D(b) = \lim_{n\longrightarrow\infty} D_{a,b}^{n}(b)$).

It follows by definition that every (bounded) 2-local derivation $ \Delta $ is a (bounded) approximately 2-local derivation.


\begin{lemma} 
Let $ \Delta $ be a 2-local ( or an approximately 2-local ) derivation of a Banach algebra 
 $ \mathcal{A} $ into a Banach $ \mathcal{A} $-bimodule $ E $. then
\begin{itemize}
\item[(i)]    $ \Delta(\lambda x) = \lambda \Delta(x) $ for any $ \lambda \in C $, and $ x \in \mathcal{A} $. 
\item[(ii)]   $ \Delta(x^{2}) =  \Delta(x)x + x\Delta(x) $ for any  $ x \in \mathcal{A} $.
\end{itemize}
\end{lemma}

\begin{proof}
We prove this lemma only for approximately 2-local derivations on Banach algebras. 
\begin{itemize} 
\item[(i)]  For each $ x \in \mathcal{A} $ and $ \lambda \in C $,  there exists a sequence of derivations
  $ \{ D_{n}^{x,\lambda x} \} $ such that 
  
  $$   \Delta(x) = \lim_{n\longrightarrow \infty}D_{n}^{x,\lambda x}(x), 
    \Delta(\lambda x) = \lim_{n\longrightarrow \infty}D_{n}^{x,\lambda x}(\lambda x) $$
  so
$$
 \Delta(\lambda x)  = \lim_{n\longrightarrow \infty}D_{n}^{x,\lambda x}(\lambda x) 
        = \lambda \hspace{.1cm} \lim_{n\longrightarrow \infty}D_{n}^{x,\lambda x}(x) = \lambda \hspace{.1cm} \Delta(x)
$$
Hence, $ \Delta $ is homogeneous.
\item[(ii)]  For each $ x \in \mathcal{A} $, there exists a class of derivations $ \{ D_{n}^{x, x^{2}} \} $ such that 
  $$ \Delta(x) = \lim_{n\longrightarrow \infty}D_{n}^{x, x^{2}}(x) \hspace{.25cm} 
\text{ and } \hspace{.25cm}
    \Delta( x^{2}) = \lim_{n\longrightarrow \infty}D_{n}^{x, x^{2}}( x^{2}) 
    $$ 
  so
 \begin{align*}
 \Delta( x^{2})  & = \lim_{n\longrightarrow \infty}D_{n}^{x, x^{2}}( x^{2}) \\
 & =  \hspace{.1cm}\big( \lim_{n\longrightarrow \infty}D_{n}^{x, x^{2}}(x)\big) x +  x \big( \lim_{n\longrightarrow \infty}D_{n}^{x, x^{2}}(x)\big)\\ 
   &  =  \hspace{.1cm} \Delta(x)x + x\Delta(x)  
 \end{align*}
\end{itemize}
\end{proof}


\begin{lemma} 
 Any additive 2-local ( or additive approximately 2-local ) derivation $ \Delta $ from a $ C^* $-algebra $ \mathcal{A} $ to
  a Banach $ \mathcal{A} $-bimodule $ E $ is a derivation. 
\end{lemma}

\begin{proof}
By applying Lemma 2.2, we conclude that each additive 2-local ( approximately 2-local ) derivation $ \Delta $ from a Banach algebra 
 $ \mathcal{A} $ to a Banach $ \mathcal{A} $-bimodule $ E $ is linear and satisfies 
 $ \Delta(x^{2}) =  \Delta(x)x + x\Delta(x) $ for any  $ x \in \mathcal{A} $, so $ \Delta $ is a jordan derivation. Now,
  it follows by (\cite{Russo},  Corollary 17) 
   that $ \Delta $ is continuous. 
  Therefore (\cite{Johnson2}, Theorem 6.3), implies that  the  continuous jordan derivation $ \Delta $ is a derivation. 
\end{proof}


\begin{definition} 
Let $ \mathscr{M} $ be a von Neumann algebra and $ \mathscr{M}_{+} $  the positive cone of $ \mathscr{M} $. A functional $ \tau $ from
  $ \mathscr{M}_{+} $ to $ [0,+\infty] $, is called a trace (on $ \mathscr{M} $) if it satisfies the following conditions:
  \begin{itemize}
  \item[(i)]   $ \tau(a+b)= \tau(a)+\tau(b) $,   $ \hspace{.25cm}   a,b \in  \mathscr{M}_{+} $; 
  \item[(ii)]  $ \tau(\lambda a)=\lambda \tau(a) $,  $ \hspace{.25cm} a \in  \mathscr{M}_{+}, \lambda \geq 0 $ 
  ( and $ 0 \cdot(+\infty) = 0 $);  
  \item[(iii)]  $ \tau(a^{*}a) = \tau(aa^{*}) $,  $ \hspace{.25cm}   a \in  \mathscr{M}_{+} $.
\end{itemize}   
A trace $ \tau $ on  $ \mathscr{M}_{+} $, is 
\begin{itemize}
\item  faithful if $ \tau(a) = 0 $ implies $ a = 0 $;   
\item  finite if $ \tau(a)< +\infty $ ($ a \in \mathscr{M}_{+} $);   
\item  semi-finite if for every non-zero $ a \in \mathscr{M}_{+} $, there exists a non-zero element $ b $ in $ \mathscr{M}_{+} $ with $ \tau(b)< +\infty $ and $ b < a $;   
\item  normal if for every bounded increasing net $ \{a_{\alpha}\} \subset \mathscr{M}_{+} $,
 $ \tau(\sup_{\alpha}  a_{\alpha}) = \sup_{\alpha} \tau(a_{\alpha}) $.  
\end{itemize} 
\end{definition}


\begin{theorem} (\cite{Sakai}, Theorem 1.7.8)  
Let $ \mathscr{M} $ be a von Neumann algebra, then the involution and left and right multiplications on $ \mathscr{M} $ are weak$^*$ continuous.   
 \end{theorem}


\begin{theorem}  
Let $ \mathscr{M} $ be a von Neumann algebra and $ \mathscr{M}_{*} $ its predual \footnote{ The set of normal linear functionals on von Neumann algebra $ \mathscr{M} $. } and let $ \tau $ be a trace on $ \mathscr{M}_{+} $. 
Let 
$ \mathscr{J}_{\tau} = \{x \in \mathscr{M} | \tau(x^{*} x) < \infty \} $ and $ \mathscr{T}_{\tau} = \mathscr{J}_{\tau}^{2} = \{ xy | x, y \in \mathscr{J}_{\tau} \} $.
Then $ \mathscr{J}_{\tau} $, $ \mathscr{T}_{\tau} $ are  two-sided ideals of $ \mathscr{M} $ and
$$ 
\mathscr{T}_{\tau} = [(\mathscr{T}_{\tau})_{+}] = \{ xy | x, y \in \mathscr{J}_{\tau} \}
 $$ 
$$
 (\mathscr{T}_{\tau})_{+} = \{ a \in \mathscr{M}_{+} | \tau (a) < \infty \}   
 $$  
where $ (\mathscr{T}_{\tau})_{+} = \mathscr{T}_{\tau} \cap \mathscr{M}_{+} $. There exists a unique linear functional  
  $ \widetilde{\tau} $ on $ \mathscr{T}_{\tau} $ which coincides with $ \tau $ on  $ (\mathscr{T}_{\tau})_{+} $, and one has  
   $ \widetilde{\tau}(ax) = \widetilde{\tau}(xa) (a \in \mathscr{T}_{\tau}, x \in \mathscr{M}) $. If $ \tau $ is normal,
   for every $ a \in \mathscr{T}_{\tau} $, the linear functional $ x \mapsto \widetilde{\tau}(ax) $ on $ \mathscr{M} $ 
    is $ \sigma(\mathscr{M}, \mathscr{M}_{*}) $-continuous.    
    Finaly $ \tau $ is semi-finite if and only if $ \mathscr{T}_{\tau} $ is $ \sigma(\mathscr{M}, \mathscr{M}_{*}) $-dense in $ \mathscr{M} $. 
\end{theorem}

\begin{proof}
See the propositions 6.5.2, 6.5.3, and 6.5.4 of refrence  \cite{Bing}.
\end{proof}

 It follows by theorem 2.6 that if $ \mathscr{M} $ is a von neumann algebra, then each finite trace $ \tau $ on  $ \mathscr{M}_{+} $ 
can be extended to a linear functional $ \widetilde{\tau} $ on $ \mathscr{M} $ such that $ \widetilde{\tau}(ab) = \widetilde{\tau}(ba) $
 ( $  a,b \in \mathscr{M} $ ) and  $ \mathscr{T}_{\tau} = \mathscr{M} $. This extention is called a tracial on $ \mathscr{M} $ and  if in addition $ \widetilde{\tau}(1)=1 $ it is called a tracial state. 
 A von neumann algebra $ \mathscr{M} $ is said to be finite if there is a faithful normal tracial state on $ \mathscr{M} $.

\section{ 2-local derivations}

\begin{definition} 
Let $ \tau: \mathcal{A}^{**}_{+} \longrightarrow [0, +\infty ] $ be a semi-finite trace on the second dual of a $ C^* $-algebra $ \mathcal{A} $, then a projection $ p \in \mathcal{A}^{**} $ is called $ \tau $-open if there is an increasing net $ (a_{\alpha}) $ of positive elements
 in $ \mathcal{A} $ with $ \tau(a_{\alpha}) < \infty $, such that $ a_{\alpha} \nearrow p $ in $ \mathcal{A}^{**} $.  
\end{definition}

\begin{theorem} 
 Let $ \mathcal{A} $ be a $ C^* $-algebra, if $ \tau $ is a faithful normal semi-finite trace  on $ \mathcal{A}^{**} $, and  
  $ 1_{ \mathcal{A}^{**} } $ is a $ \tau $-open projection, then each 2-local derivation $ \Delta $ from $ \mathcal{A} $ to $ \mathcal{A}^{**} $ 
  is an inner derivartion.   
\end{theorem}

\begin{proof}  
Let $ \Delta:\mathcal{A}\longrightarrow \mathcal{A}^{**} $ be an 2-local derivation. By Theorem 2.6, $ \tau $ can be extended uniquely to a linear functional $ \widetilde{\tau} $ on $ \mathscr{T}_{\tau} $, such that 
$ \widetilde{\tau}(ax) = \widetilde{\tau}(xa) $ ( $ a \in \mathscr{T}_{\tau}, x \in \mathcal{A}^{**} $ ).
  For each $ x \in \mathcal{A} $ and 
 $ y \in \mathscr{T}_{\tau} \cap \mathcal{A} $ there exists a derivation  $ D_{x,y} $ from  $ \mathcal{A} $ 
to $ \mathcal{A}^{**} $, such that $ \Delta(x) = D_{x,y}(x)$ and 
     $ \Delta(y) = D_{x, y}(y)$.  It follows by Lemma 2.1 that there is an element $ m \in \mathcal{A}^{**} $ such that 
$$
[m, xy] = D_{x, y}(xy) = D_{x, y}(x)y + xD_{x, y}(y), 
$$
so
$$
[m, xy] = D_{x, y}(x)y + xD_{x, y}(y). 
$$ 

Since $ \mathscr{T}_{\tau} $ is an ideal and $ y \in \mathscr{T}_{\tau} \cap \mathcal{A} $, the elements $ axy $, $ xy $, $ xya $
 and $ \Delta(y) $ also belong to $ \mathscr{T}_{\tau} $ and hence by continuity of $ \widetilde{\tau} $ we have
\begin{align*}
\widetilde{\tau}(mxy) &= \widetilde{\tau}((mx)y) = \widetilde{\tau}(y(mx)) \\ 
          &= \widetilde{\tau}((ym)x) = \widetilde{\tau}(x(ym)) \\
         & = \widetilde{\tau}(xym).
\end{align*}
Therefore
$$  
\widetilde{\tau}(D_{x, y}(x)y + xD_{x, y}(y)) = \widetilde{\tau} ([m, xy]) = \widetilde{\tau}(mxy - xym) = 0,
$$
 so 
$$
\widetilde{\tau}(D_{x, y}(x)y) = -\widetilde{\tau}(xD_{x, y}(y)). 
$$
Based the above analysis, the following equation can be obtained 
$$
\widetilde{\tau}(D_{x, y}(x)y) = - \widetilde{\tau}(xD_{x, y}(y)), 
$$
so 
$$
\widetilde{\tau}(\Delta(x)y) = - \widetilde{\tau}(x\Delta(y)). 
$$ 
  For arbitrary $ u,v \in \mathcal{A} $ and $ w \in \mathscr{T}_{\tau} \cap \mathcal{A} $ set $ x = u+v $, $ y=w $. Then
   $ \Delta(w) \in \mathscr{T}_{\tau} $ and
\begin{align*}
\widetilde{\tau}(\Delta(u + v)w) &=  - \widetilde{\tau}((u + v)\Delta(w)) = -\widetilde{\tau} (u\Delta(w)) -\widetilde{\tau} (v\Delta(w)) \\
    & = \widetilde{\tau}(\Delta(u)w) + \widetilde{\tau}(\Delta(v)w) = \widetilde{\tau}((\Delta(u) + \Delta(v))w), 
\end{align*}
so for all $ u, v \in \mathcal{A} $ and $ w \in \mathscr{T}_{\tau} \cap \mathcal{A} $ we have
$$
\widetilde{\tau}((\Delta(u + v) - \Delta(u) - \Delta(v))w) = 0. 
$$
 Put $ b = \Delta(u + v) - \Delta(u) - \Delta(v) $. Then 
\begin{equation*}
 (*) \hspace{1cm} \widetilde{\tau}(bw)=0 \hspace{.5cm} ( w \in \mathscr{T}_{\tau} \cap \mathcal{A} )  
\end{equation*} 
Now by hypothesis, there is a net $ \{ e_{\alpha} \}_{\alpha} $ of elements in 
 $ \mathscr{T}_{\tau} \cap \mathcal{A}_{+} $ such that  $ e_{\alpha} \uparrow 1 $ in $ \mathcal{A}^{**} $, 
  and  there is a net $ \{b_{\beta}\}_{\beta} $ in $ \mathcal{A} $ such that $ b_{\beta}\longmapsto b^{*} $ in $ \sigma(\mathcal{A}^{**}, \mathcal{A}^{*}) $-top.
Then $ \{ e_{\alpha}b_{\beta} \}_{ \beta} \subset \mathscr{T}_{\tau} \cap \mathcal{A} $ for each $ \alpha $. Hence,  $ (*) $
 and theorem 2.6, implies that  
$$ 
 \tau( be_{\alpha}b^{*} ) = \widetilde{\tau}( be_{\alpha}b^{*} ) 
 = w^{*}-\lim_{\beta}\widetilde{\tau}( be_{\alpha}b_{\beta} ) = 0 \hspace{.5cm} ( \text{ for all } \alpha ).
$$   
 Since the trace $ \tau $ is normal and $ be_{\alpha}b^{*} \uparrow bb^{*} $ in $ \mathcal{A}^{**} $, we have 
 $ \tau( be_{\alpha}b^{*} ) \uparrow \tau(bb^{*}) $,
so $ \tau(bb^{*}) = 0 $. Since $ \tau $ is faithful we have $ bb^{*} = 0 $, i.e. $ b = 0 $. Therefore 
$$
\Delta(u + v) = \Delta(u) + \Delta(v), \hspace{.5cm} u, v \in \mathcal{A}.
$$
Thus $ \Delta $ is an additive 2-local derivation. It follows by lemma 2.3 that the linear operator $ \Delta $ is a derivation. 
\end{proof}


\begin{proposition}  
 Let $ C^{*} $-algebra $ \mathcal{A} $ be an ideal in its second dual $ \mathcal{A}^{**} $  
and let $ \tau $ be a normal semi-finite trace on $ \mathcal{A}^{**} $ and let the restriction $ \tau|_{\mathcal{A}} $ on $ \mathcal{A} $ be faithful. Then each 2-local derivation $ \Delta $  on $  \mathcal{A} $ is a derivation.  
\end{proposition}

\begin{proof} 
Let $ \Delta:\mathcal{A}\longrightarrow \mathcal{A} $ be a 2-local derivation. For each $ x \in \mathcal{A} $ and 
 $ y \in \mathscr{T}_{\tau} \cap \mathcal{A} $ there is a derivation $ D_{x,y} $ on $ \mathcal{A} $ such that $ \Delta(x) = D_{x,y}(x) $ and 
  $ \Delta(y) = D_{x,y}(y) $. It follows by ( \cite{Sakai}, Corollary 4.1.7 ) that, there is an element 
$ a \in \mathcal{A}^{**} $ such that
$$
[a, xy] = D_{x, y}(xy) = D_{x, y}(x)y + xD_{x, y}(y).  
$$
Same as theorem 3.2, it can be shown that  
$$
\widetilde{\tau}((\Delta(u + v) - \Delta(u) - \Delta(v))w) = 0,
$$
for all $ u, v \in \mathcal{A} $ and $ w \in \mathscr{T}_{\tau} \cap \mathcal{A} $ . Put $ b = \Delta(u + v) - \Delta(u) - \Delta(v) $, Then
 $$ 
 \widetilde{\tau}(bw)=0 \hspace{.5cm} ( w \in \mathscr{T}_{\tau} \cap \mathcal{A} )
 $$
 Now take an increasing net $ \{ e_{\alpha} \}_{\alpha} $ of projections in
 $ \mathscr{T}_{\tau} $ such that  $ e_{\alpha} \uparrow 1 $ in $ \mathcal{A}^{**} $ ( see \cite{Sakai}, Theorem 2.5.6. ). 
Since $ \mathcal{A}$ and $ \mathscr{T}_{\tau} $ are ideals of $ \mathcal{A}^{**} $ and $ b^{*} \in \mathcal{A} $, we have
 $ \{ e_{\alpha}b^{*} \}_{\alpha} \subset \mathscr{T}_{\tau} \cap \mathcal{A} $. 
Hence  
$$
\tau( be_{\alpha}b^{*} ) = \widetilde{\tau}( be_{\alpha}b^{*} ) = 0, \hspace{.5cm} ( \text{ for all } \alpha ).
$$ 
At the same time $ be_{\alpha}b^{*} \uparrow bb^{*} $ in $ \mathcal{A}^{**} $. Since  $ \tau $ is normal, we have 
 $ \tau( be_{\alpha}b^{*} ) \uparrow \tau(bb^{*}) $ and so
 $ \tau(bb^{*}) = 0 $. Since the restriction $ \tau|_{\mathcal{A}} $ is faithful, we have $ bb^{*} = 0 $ and so $ b = 0 $. Therefore
$$ 
\Delta(u + v) = \Delta(u) + \Delta(v), \hspace{.5cm} ( u, v \in \mathcal{A} ).
$$
Thus $ \Delta $ is an additive 2-local derivation. It follows by lemma 2.3 that the linear operator 
$ \Delta $ is a derivation. 
\end{proof}

\section{Approximately 2-Local Derivations}

\begin{lemma}  
Let $ \tau $ be a faithful normal tracial state on the second dual $ \mathcal{A}^{**} $ of a $ C^* $-algebra $ \mathcal{A} $ and let $ b\in \mathcal{A}^{**} $ be arbitrary. 
Then $ \tau(xb) = 0 $ ( $ x \in \mathcal{A} $ ), implies that $ b = 0 $.  
\end{lemma}

\begin{proof}
Let $ b\in \mathcal{A}^{**} $ then there is a net $ (b_{\alpha}) $ in $ \mathcal{A} $ such that $ b_{\alpha}$ converges to $ b^{*} $ in 
$ \sigma(\mathcal{A}^{**}, \mathcal{A}^{*}) $-top, and so
$$
\tau(b^{*}b) =w^{*}-\lim_{\alpha}\tau(b_{\alpha}b) = 0. 
$$
Since $ \tau $ is faithful, we have $ b = 0 $.  
\end{proof}


\begin{theorem}  
Let the second dual $ \mathcal{A}^{**} $ of a $ C^* $-algebra $ \mathcal{A} $ be finite as a von neumann algebra.
 Then each approximately 2-local derivation $ \Delta $ from $ \mathcal{A} $ to $  \mathcal{A}^{**} $, is a derivation.   
\end{theorem}

\begin{proof}
Let $ \Delta $ from $ \mathcal{A} $ to $  \mathcal{A}^{**} $ is an approximately 2-local derivation and let
  $ \tau:\mathcal{A}^{**}\longrightarrow C $ be a faithful normal tracial state. For each $ x, y \in \mathcal{A} $ there exists a sequence of derivations  $\{D^{n}_{x,y}\} $ from  $ \mathcal{A} $  
to $ \mathcal{A}^{**} $, such that $ \Delta(x) = \lim_{n\longrightarrow\infty} D^{n}_{x,y}(x)$ and 
     $ \Delta(y) = \lim_{n\longrightarrow\infty} D_{x, y}^{n}(y)$.  It follows by Lemma 2.1  that $ D^{n}_{x, y} $ is inner ( for each
      $ n\in \mathbb{N} $ ), so there exists an element $ m \in \mathcal{A}^{**} $ such that 
$$
[m, xy] = D_{n}^{x, y}(xy) = D_{n}^{x, y}(x)y + xD_{n}^{x, y}(y),  
$$  
so
$$
[m, xy] = D_{n}^{x, y}(x)y + xD_{n}^{x, y}(y). 
$$
Based on that $ \tau $ is continuous  
\begin{align*}
\tau(mxy) &= \tau((mx)y) = \tau(y(mx)) \\
          &= \tau((ym)x) = \tau(x(ym)) \\
         & = \tau(xym).
\end{align*}
Therefore
$$
\tau(D_{n}^{x, y}(x)y + xD_{n}^{x, y}(y)) = \tau ([m, xy]) = \tau(mxy - xym) = 0,   
$$
so
$$
\tau(D_{n}^{x, y}(x)y) = -\tau(xD_{n}^{x, y}(y)). 
$$
Based the above analysis, the following equation can be obtained
$$
 \lim_{n\longrightarrow +\infty}\tau(D_{n}^{x, y}(x)y) = - \lim_{n\longrightarrow +\infty}\tau(xD_{n}^{x, y}(y)), 
$$
so
$$
\tau(\Delta(x)y) = - \tau(x\Delta(y)). 
$$

For arbitrary $u, v, w \in \mathcal{A} $, set $x = u + v$, $y = w$. Then from above we obtain
\begin{align*}
\tau(\Delta(u + v)w) &= -\tau((u + v)\Delta(w)) \\
                   & = -\tau(u\Delta(w)) - \tau(v\Delta(w)) \\
                   & = \tau(\Delta(u)w) + \tau(\Delta(v)w) \\
                   & = \tau((\Delta(u) + \Delta(v))w). 
\end{align*}  
Hence
\begin{equation*}
\tau((\Delta(u + v) - \Delta(u) - \Delta(v))w) = 0.
\end{equation*}
For all $u, v, w \in \mathcal{A}$, and by Lemma 4.1,  we have
 $$ 
 \Delta(u + v) - \Delta(u) - \Delta(v) = 0, 
 $$
 so
\begin{equation*} 
\Delta(u + v) = \Delta(u) +  \Delta(v). 
\end{equation*}
By the analysis, we obtain $ \Delta $  is an additive mapping. 
 Hence, 
$ \Delta $ is an additive 2-local derivation, and  lemma 2.3 implies that $ \Delta $ is a derivation. 
\end{proof}

\begin{problem} 
Does Theorem 3.2 hold for approximately 2-local derivations?  
\end{problem}


\end{document}